\documentclass[12pt]{article}

\usepackage{amsmath,amsthm,amsfonts,amssymb,amscd,cite}
\usepackage{graphicx}

\allowdisplaybreaks[4]

\newtheorem {lemma}{Lemma}[section]
\newtheorem {theorem} {Theorem}[section]

\allowdisplaybreaks [4]
\usepackage{fullpage}

\begin{document}

\title{\bf Distribution of signless Laplacian eigenvalues and graph invariants}

\author{Leyou Xu\footnote{E-mail: leyouxu@m.scnu.edu.cn},
Bo Zhou\footnote{E-mail: zhoubo@m.scnu.edu.cn}\\
School of Mathematical Sciences, South China Normal University\\
Guangzhou 510631, P.R. China}

\date{}
\maketitle

\begin{abstract}
For a simple graph on $n$ vertices, any of its signless Laplacian eigenvalues is in the interval $[0, 2n-2]$. In this paper, we give relationships between the number of signless Laplacian eigenvalues in specific intervals in $[0, 2n-2]$ and graph invariants including matching number and diameter.
\end{abstract}

\section{Introduction}

%In addition, all the graphs considered in this paper are undirected and simple.
All graphs considered in this paper are simple and finite.
Let $G$ be a graph with vertex set $V(G)$ and edge set $E(G)$.
For $v\in V(G)$, denote by $N_G(v)$ the neighborhood of $v$ in $G$, and  $\delta_G(v)$ denotes the degree of $v$ in $G$.  Denote by $\delta(G)$ the minimum degree of $G$.
For an $n$-vertex graph $G$, the signless Laplacian matrix of  $G$  is the $n\times n$ matrix
$Q(G)=(q_{uv})_{u,v\in V(G)}$, where
\[
q_{uv}=\begin{cases}
\delta_G(u) & \mbox{if $u=v$,}\\
1     & \mbox{if $u\ne v$ and $uv\in E(G)$,}\\
0     & \mbox{if $u\ne v$ and $uv\not\in E(G)$}.
\end{cases}
\]
That is, $Q(G)$ is equal to the sum of the  diagonal
degree matrix and the adjacency matrix of $G$.
The  eigenvalues of $Q(G)$ are known as the signless Laplacian eigenvalues of $G$, which we denote by $q_1(G), \dots, q_n(G)$, arranged in nonincreasing order. That is,
$q_j(G)$ is the $j$-th (largest) signless Laplacian eigenvalue  of $G$ for $j=1,\dots, n$.
It is known that $Q(G)$ is a positive semidefinite matrix, so $q_n(G)\ge 0$. By Gershgorin's circle theorem,
$q_1(G)\le 2n-2$.  It  is interesting to know how the signless Laplacian eigenvalues are distributed in $[0,2n-2]$. To have a fuller understanding of the  distribution of the signless Laplacian eigenvalues of  a graph of a given order,  Ghodrati and Hosseinzadeh \cite{GH} proposed to investigate  the number of signless Laplacian eigenvalues of an $n$-vertex graph in some subintervals of $[0, 2n-2]$. To this end, denote by $m_GI$ for a subinterval $I$ of $[0,2n-2]$ the number of signless Laplacian eigenvalues of  an $n$-vertex graph $G$ that fall inside $I$ (counting multiplicities), which,  for specific $I$,  will be
related to structural properties of $G$.
Wang and Belardo \cite{WB2}  determined all connected graphs with at most two signless Laplacian eigenvalues exceeding two.
Lin and Zhou \cite{LZ} determined all connected graphs with at most one signless Laplacian eigenvalue exceeding three. Some results on the distribution of the signless Laplacian eigenvalues are in fact known
from the bounds on the signless Laplacian eigenvalues. For example,
if $G$ is a graph on $n\ge 2$ vertices, then $q_2(G)\le n-2$ (which follows from Lemma \ref{addedges} below, see also \cite{LN,WBHB}), so $m_G[0,n-2]\ge n-1$, and  moreover, if $G$ is not complete, then $q_2(G)\ge \delta(G)$  (see \cite{LN}), so $m_G[0, \delta(G))\le n-2$.
Ghodrati and Hosseinzadeh \cite{GH} established some relationship between
the number of signless Laplacian eigenvalues in certain subintervals of $[0,2n-2]$ and graph parameters including independence, clique, chromatic,  edge covering and matching numbers.

The matching number of a graph $G$ is the size of a maximum matching of $G$, and it is denoted by $\nu(G)$.
For a connected graph $G$ with $u,v\in V(G)$,
the distance between $u$ and $v$ in $G$ is the length of a
shortest path connecting them in $G$.  The diameter of  $G$ is the greatest distance between vertices in $G$.

In this paper, we show relationships between the number of signless Laplacian eigenvalues within specific intervals for $n$-vertex graphs and  matching number (diameter, respectively).

The Laplacian eigenvalues of a graph $G$ are the eigenvalues of the Laplacian matrix of $G$. The Laplacian spectrum  of any $n$-vertex graph is contained in $[0,n]$. The distribution of Laplacian eigenvalues received much attention, see, e.g. \cite{AADT, XZ1}. For bipartite graphs, the Laplacian spectrum and the signless Laplacian spectrum coincide. However, there is no direct connection between the
distribution of Laplacian eigenvalues and the distribution of signless Laplacian eigenvalues for non-bipartite graphs.

\section{Preliminaries}

For an $n\times n$ complex matrix $M$ whose all eigenvalues are real, $\rho_i(M)$ denotes its $i$-th largest eigenvalue for $i=1,\dots, n$, and $\sigma(M)=\{\rho_i(M): i=1,\dots,n \}$ is the spectrum of $M$.
For convenience, if $\rho$ is an eigenvalue of $M$ with multiplicity $s\ge 2$, then we write it as $\rho^{[s]}$ in $\sigma(M)$. For a graph $G$, let $\sigma_Q(G)=\sigma(Q(G))$.
Denote by $I_n$ the identity matrix of order $n$.

We need Weyl's inequalities (Lemma \ref{cw}) and Cauchy's interlacing theorem (Lemma \ref{interlacing}).

\begin{lemma} \cite[Theorem 1.3]{So} \label{cw}
Let $A$ and $B$ be Hermitian matrices of order $n$.
For $1\le i,j\le n$ with $i+j-1\le n$,
\[
\rho_{i+j-1}(A+B)\le \rho_i(A)+\rho_j(B).
 \]
\end{lemma}

\begin{lemma}\label{interlacing}\cite[Theorem 4.3.28]{HJ}
If $M$ is  a Hermitian matrix of order $n$ and $B$ is one of its principal submatrices of order $p$, then $\rho_{n-p+i}(M)\le\rho_i(B)\le \rho_{i}(M)$ for $i=1,\dots,p$.
\end{lemma}

For a graph $G$ with $\emptyset\ne S\subseteq V(G)$, denote by $G[S]$ the subgraph of $G$ induced by $S$ and let $G-S=G[V(G)\setminus S]$ if $S\ne V(G)$.

For a graph $G$ with $S\subseteq E(G)$, denote by $G-S$ the subgraph of $G$ obtained from $G$ by deleting all edges in $S$.
Particularly, if $S=\{e\}$, then we write it as $G-e$.
If $G'=G-S$ for some $S\subseteq E(G)$, then $G=G'+S$.

Applying Lemma \ref{cw},  interlacing theorems for signless Laplacian eigenvalues of a graph when an edge or a vertex is deleted have been established, which are crucial for our main results.

\begin{lemma}\label{addedges}\cite{CRS}
If $G$ is an $n$-vertex graph with $e\in E(G)$, then
\[
q_1(G)\ge q_1(G-e)\ge q_2(G)\ge \dots\ge q_n(G)\ge  q_{n}(G-e).
\]
\end{lemma}

\begin{lemma}\label{deletevertices}\cite{WB}
If $G$ is an $n$-vertex graph with $v\in V(G)$, then $q_{i+1}(G)\le q_i(G-v)+1$ for $i=1,\dots,n-1$.
\end{lemma}

Let $G$ be a graph. Suppose that $V(G)$ is partitioned as $V_1\cup \dots\cup V_m$. For $1\le i<j\le m$, set $Q_{i,j}$ to be the submatrix of $Q(G)$ with rows corresponding to vertices in $V_i$ and columns corresponding to vertices in $V_j$. The quotient matrix of $Q(G)$ with respect to the partition $V_1\cup \dots \cup V_m$ is denoted by $B=(b_{ij})$, where $b_{ij}=\frac{1}{|V_i|}\sum_{u\in V_i}\sum_{v\in V_j}q_{uv}$. If $Q_{i,j}$ has constant row sum, then we say $B$ is an equitable quotient matrix (with respect to the above partition of $V(G)$).
The following lemma is an immediate consequence of \cite[Lemma 2.3.1]{BH}

\begin{lemma}\label{quo}
For a graph $G$, if $B$ is an equitable quotient matrix of $Q(G)$, then $\sigma(B)\subseteq \sigma_Q(G)$.
\end{lemma}

Let $G\cup H$ be the disjoint union of graphs $G$ and $H$. The disjoint union of $k$ copies of a graph $G$ is denoted by $kG$.
Denote by $P_n$, $C_n$ and $K_n$ the path, the cycle and the complete graph of order $n$, respectively.
Denote by $K_{2,n}$ the complete bipartite graph with two and $n$ vertices in its partite sets.

For integers $n$, $d$, $r$ and $a$ with $3\le d\le n-2$, $2\le r\le d-1$ and $1\le a\le n-d-2$,  let $G_{n,d,r,a}$ be the graph obtained from vertex disjoint union of $P_{d+1}:=v_1\dots v_{d+1}$ and  $K_{n-d-1}$ by adding all edges between  $a$ vertices in $K_{n-d-1}$ and vertices $v_{r-1}$, $v_r$ and $v_{r+1}$ and between the remaining $n-d-1-a$ vertices in $K_{n-d-1}$ and  vertices $v_r$, $v_{r+1}$ and $v_{r+2}$.

A diametral
path of a connected graph $G$ is a shortest path between vertices whose distance from each other is the
diameter of $G$.

\begin{lemma}\label{Cn}
(i) $\sigma_Q(C_n)=\left\{2+2\cos \frac{2j\pi}{n}:j=0,\dots,n-1 \right\}$.\\
(ii) $q_1(P_n)<4$.\\
(iii) For $n\ge 5$ and $e\in E(K_n)$, $q_{n-3}(K_n-e)=q_{n-2}(K_n-e)=n-2$.\\
(iv) $q_3(G_{6,3,2,1})=q_4(G_{6,3,2,1})=3$.\\
(v) For $n\ge 7$ and $1\le a\le n-5$, $q_{n-4}(G_{n,3,2,a})=q_{n-3}(G_{n,3,2,a})=q_{n-2}(G_{n,3,2,a})=n-3$ and $q_{n-1}(G_{n,3,2,a})<n-3$.
\end{lemma}

\begin{proof} From \cite[pp. 8--9]{BH}, the spectrum of (the adjacency matrix of) $C_n$ is
$\left\{2\cos \frac{2j\pi}{n}:j=0,\dots,n-1 \right\}$, from which Part (i) follows.
Part (ii) follows from the Laplacian spectrum of $P_n$ given in \cite[p. 9]{BH}, as it is also the
signless Laplacian spectrum of $P_n$.

For $(n-2)I_n-Q(K_n-e)$, the two rows corresponding the end vertices of $e$ are equal and the remaining
$n-2$ rows are also equal, so the rank of $(n-2)I_n-Q(K_n-e)$ is at most two, and
$n-2$ is an eigenvalue of $Q(K_n-e)$ with multiplicity at least $n-2$.
Let $W_1$ be the set of vertices of degree $n-1$ in $K_n-e$ and $W_2=V(K_n-e)\setminus W_1$.
With respect to the partition $V(K_n-e)=W_1\cup W_2$, $Q(K_n-e)$ has an equitable quotient matrix $B$ with \[
B=\begin{pmatrix}
2n-4 &2\\
n-2&n-2
\end{pmatrix}.
\]
Let $f(x):=\det (xI_2-B)=x^2-(3n-6)x+2n^2-10n+12$. The roots of $f(x)=0$ are $x=r,s$,  where
\[
r=\frac{3}{2}n-3 +\frac{1}{2}\sqrt{n^2+4n-12},
\]
\[
s=\frac{3}{2}n-3 -\frac{1}{2}\sqrt{n^2+4n-12}.
\]
By Lemma \ref{quo}, $\sigma_Q(K_n-e)=\{(n-2)^{[n-2]}, r, s\}$.
It is easy to see that $r>n-2>s$, so $q_{n-3}(K_n-e)=q_{n-2}(K_n-e)=n-2$. %and $q_{n-3}=n-2$ if $n\ge 5$.
This proves Part (iii).

Part (iv) follows from an easy calculation.
% with
%\[
%Q(G_{6,3,1,1})=\begin{pmatrix}
%2&1&1&0&0&0\\
%1&4&1&1&1&0\\
%1&1&4&1&1&0\\
%0&1&1&4&1&1\\
%0&1&1&1&4&1\\
%0&0&0&1&1&2
%\end{pmatrix}.
%\]

Let $G:=G_{n,3,2,a}$.
Assume that $a\le n-4-a$.
Let $v_1\dots v_4$ be a diametral path of $G$ so that $\delta_G(v_1)=a+1$ and $\delta_G(v_4)=n-3-a$.
As there are $a+1$ and $n-3-a$ equal rows in  $(n-3)I_n-Q(G)$, $n-3$ is an eigenvalue of $Q(G)$ with multiplicity at least $n-4$.
With respect to the partition $V(G)=\{v_1\}\cup N_G(v_1)\cup N_G(v_4)\cup \{v_4\}$, $Q(G)$ has an equitable quotient matrix $M$ with \[
M=\begin{pmatrix}
a+1&a+1&0&0\\
1&n-2+a&n-3-a&0\\
0&a+1&2n-6-a&1\\
0&0&n-3-a&n-3-a
\end{pmatrix}.
\]
Let $g(x):=\det(xI_4-M)$.
Suppose first that $a<n-4-a$.
Note that as a quartic function on $x$, $g(x)$ satisfies
\[
g(n-2)=4a(n-4-a)-(n-4)^2<0,
\]
\[
g(n-3)=(n-3)((n-4-a)a+n-3)>0,
\]
and
\[
g(a+1)=-(a+1)(2(n-4-a)(n-2a-4)-a-1)<0.
\]
It follows that
\[
\rho_4< a+1<\rho_3<n-3<\rho_2<n-2<\rho_1,
\]
where  $\rho_i=\rho_i(M)$ for $i=1,\dots, 4$.
By Lemma \ref{quo},  $\sigma_Q(G)=\{\rho_1, \rho_2,(n-3)^{[n-4]}, \rho_3, \rho_4\}$. As $n\ge 7$, we have  $q_{n-4}(G)=q_{n-3}(G)=q_{n-2}(G)=n-3$ and $q_{n-1}(G)=\rho_3<n-3$, as desired.
Suppose next that $a=n-4-a$, i.e., $a=\frac{n-4}{2}$.
Then $g(n-2)=0$ and $g(a)=0$.  By Lemma \ref{quo},
$n-2$ and $a$ are signless Laplacian eigenvalues of $G$.
Let $V_1=\{v_1, v_4\}$ and $V_2=N_G(v_1)\cup N_G(v_4)$. Then,
with respect to the partition $V(G)=V_1\cup V_2$, $Q(G)$ has an equitable quotient matrix $M'$ with
\[
M'=\begin{pmatrix}
a+1&a+1\\
1&n-1+2a
\end{pmatrix},
\]
whose characteristic polynomial is $h(x):=x^2-(5a+4)x+(4a+2)(a+1)$.
As
$h(a+1)=-(a+1)<0$
and
$h(a)=2a+2>0$,
we have $a<\rho_2'<a+1<\rho_1'$, where $\rho_i'=\rho_i(M')$ for $i=1,2$.
As $\rho_1'\ge \frac{5a+4}{2}>n-2$, we have
\[
a<\rho_2'<a+1<n-3<n-2<\rho_1'.
\]
 By Lemma \ref{quo} again, $\sigma_Q(G)=\{\rho_1', n-2, (n-3)^{[n-4]}, \rho_2', a\}$. So $q_{n-4}(G)=q_{n-3}(G)=q_{n-2}(G)=n-3$ and $q_{n-1}(G)=\rho_2'<n-3$.
 This proves Part (v).
\end{proof}

For integers $n$, $d$ and $t$ with $2\le d\le n-2$ and $2\le t\le d$, let $G_{n,d,t}$ be the graph  obtained from vertex disjoint union of $P_{d+1}:=v_1\dots v_{d+1}$ and  $K_{n-d-1}$  by adding all edges between vertices of $K_{n-d-1}$ and vertices $v_{t-1}$, $v_t$ and $v_{t+1}$.

\begin{lemma}\label{gndt}
For integers $n$, $d$ and $t$ with $2\le t\le d\le n-3$,
$m_{G_{n,d,t}}[0,n-d+1)\ge d$.
\end{lemma}

\begin{proof} Let $G=G_{n,d,t}$.
Let $P:=v_1\dots v_{d+1}$ be a diametral path of $G$.
If $d=2$, then $G\cong K_n-v_1v_3$, so the result follows from Lemma \ref{Cn} (iii). Suppose that $d\ge 3$. It suffices to show that $q_{n-d+1}(G)<n-d+1$.

Suppose first that $t=2$. Then $G-v_{3}v_{4}\cong (K_{n-d+2}-e)\cup P_{d-2}$. By Lemma \ref{addedges},
 \begin{align*}
q_{n-d+1}(G)&\le q_{n-d}(G-v_{3}v_{4})=q_{n-d}((K_{n-d+2}-e)\cup P_{d-2})\\
&\le \max\{q_{n-d}(K_{n-d+2}-e),q_1(P_{d-2})\}.
\end{align*}
By Lemma \ref{Cn}, $q_{n-d}(K_{n-d+2}-e)= n-d$ and $q_1(P_{d-2})<4$. If $d\le n-4$, then
$\max\{n-d,q_1(P_{d-2})\}=n-d<n-d+1$. If $d=n-3$, then $\max\{n-d,q_1(P_{d-2})\}<4=n-d+1$.
So
\[
q_{n-d+1}(G)\le \max\{q_{n-d}(K_{n-d+2}-e),q_1(P_{d-2})\}<n-d+1.
\]
Similar argument applies to the case when $t=d$.
Suppose now that $3\le t\le d-1$. Then  $G-v_{t-2}v_{t-1}-v_{t+1}v_{t+2}\cong  (K_{n-d+2}-e)\cup P_{t-2}\cup P_{d-t}$. By Lemmas \ref{addedges} and \ref{Cn},
\begin{align*}
q_{n-d+1}(G) & \le q_{n-d-1}(G-v_{t-2}v_{t-1}-v_{t+1}v_{t+2})\\
&=q_{n-d-1}(P_{t-2}\cup (K_{n-d+2}-e)\cup P_{d-t})\\
&\le \max \{q_{n-d-1}(K_{n-d+2}-e),q_1(P_{t-2}),q_1(P_{d-t})\}\\
&<n-d+1. \qedhere
\end{align*}
\end{proof}

\begin{lemma}\label{gndra}
For integers $n$ and $t$ with $n\ge 6$ and $2\le t\le n-4$, $q_5(G_{n,n-3,t,1})<4$.
\end{lemma}
\begin{proof}
Let $G=G_{n,n-3,t,1}$ and let $P:=v_1\dots v_{n-2}$ be a diametral path of $G$.

If $n=6$, then the result follows directly from Lemma \ref{Cn}.
Suppose that $n\ge 7$.
Let $G=G_{n,n-3,t,1}$.
If $t=2$, then $G-v_{4}v_{5}\cong G_{6,3,2,1}\cup P_{n-6}$, so by Lemmas \ref{addedges} and \ref{Cn}, \[
q_5(G)\le q_4(G-v_4v_5)\le \max \{q_4(G_{6,3,2,1}),q_1(P_{n-6})\}<4.
\]
If $t=n-4$, the above argument applies by replacing  $v_{4}v_{5}$ by $v_{n-6}v_{n-5}$.
%If $t=n-4$, then $G-v_{n-6}v_{n-5}\cong P_{n-6}\cup G_{6,3,1,1}$ and so by Lemmas \ref{addedges} and \ref{Cn} again, \[
%q_5(G)\le q_4(G-v_{n-6}v_{n-5})\le \max \{q_4(G_{6,3,1,1}),q_1(P_{n-6})\}<4.
%\]
Suppose  that $3\le t\le n-5$. As $G-v_{t-2}v_{t-1}-v_{t+2}v_{t+3}\cong P_{t-2}\cup G_{6,3,2,1}\cup P_{n-t-4}$, we have by Lemmas \ref{addedges} and \ref{Cn}, \[
q_5(G)\le q_3(G-v_{t-2}v_{t-1}-v_{t+2}v_{t+3})\le \max \{q_1(P_{t-2}),q_3(G_{6,3,3,1}),q_1(P_{n-t-4}) \}<4. \qedhere
\]
\end{proof}

\begin{lemma}\label{gndr}
For integers $n$, $d$, $t$ and $a$ with $2\le t\le d-1\le n-4$ and $1\le a\le n-d-2$, $m_{G_{n,d,t,a}}[0,n-d+1)\ge d$.
\end{lemma}
\begin{proof}
Let $G=G_{n,d,t,a}$ and let $P:=v_1\dots v_{d+1}$ be a diametral path of $G$.

If $d=3$, then the result follows directly from Lemma \ref{Cn}. Suppose that $d\ge 4$. Let $G=G_{n,d,t,a}$.
If $t=2$, then $G-v_4v_5\cong G_{n-d+3,3,2,a}\cup P_{d-3}$, so by Lemmas \ref{addedges} and \ref{Cn}, \[
q_{n-d+1}(G)\le q_{n-d}(G-v_4v_5)\le \max\{q_{n-d}(G_{n-d+3,3,2,a}),q_1(P_{d-3})\}.
\]
As $n-d\ge 5$, we have by Lemma \ref{Cn} that $q_{n-d}(G_{n-d+3,3,2,a})=n-d>4>q_1(P_{d-3})$.
So
\[
q_{n-d+1}(G)\le\max\{q_{n-d}(G_{n-d+3,3,2,a}),q_1(P_{d-3})\}=n-d<n-d+1.
\]
Similar argument applies to the case when $t=d-1$.
Suppose that $3\le t\le d-2$. Then $G-v_{t-2}v_{t-1}-v_{t+2}v_{t+3}\cong G_{n-d+3,3,2,a}\cup P_{t-2}\cup P_{d-t-1}$. By Lemmas \ref{addedges} and \ref{Cn},
\begin{align*}
q_{n-d+1}(G)&\le q_{n-d-1}(G-v_{t-2}v_{t-1}-v_{t+2}v_{t+3})\\
&=q_{n-d-1}(G_{n-d+3,3,2,a}\cup P_{t-2}\cup P_{d-t-1})\\
&\le \max \{q_{n-d-1}(G_{n-d+3,3,2,a}), q_1(P_{t-2}),q_1(P_{d-t-1}) \}\\
&<n-d+1.\qedhere
\end{align*}
\end{proof}

Given a graph $G$, for $v\in V(G)$ and a subgraph $F$ of $G$, let
$N_{G,F}(v)=N_G(v)\cap V(F)$ and $\delta_{G,F}(v)=|N_{G,F}(v)|$.

\section{Main results}

In this section, we provide some bounds for the number of signless Laplacian eigenvalues within specific intervals for $n$-vertex graphs with given parameters, such as matching number and diameter.

\subsection{Distribution of signless Laplacian eigenvalues and matching number}

We need three lemmas.

\begin{lemma}\label{bcn}
For $n\ge 3$ with $n\ne 5$, $m_{C_n}[0,1)\le \nu(C_n)-1$.
\end{lemma}

\begin{proof}
By Lemma \ref{Cn},
\[
m_{C_n}[0,1)=\begin{cases}
\lceil \frac{n}{3}\rceil &\mbox{ if }n\equiv 2  \pmod 3,\\
\lceil \frac{n}{3}\rceil-1 &\mbox{ if }n\equiv 0,1 \pmod  3.
\end{cases}
\]
As $\nu(C_n)=\lfloor\frac{n}{2}\rfloor$, $m_{C_n}[0,1)\le \nu(C_n)-1$, as desired.
\end{proof}

\begin{lemma}\label{k2t}
(i) For $n\ge 4$, $m_{K_{2,n-2}}[0,1)=\nu(K_{2,n-2})-1$.\\
(ii) $m_{P_6}[0,1)=\nu(P_6)-1$.
\end{lemma}
\begin{proof}
As $\sigma_Q(K_{2,n-2})=\{n,n-2,2^{[n-3]},0\}$, $m_{K_{2,n-2}}[0,1)=1$.
So the result follows by noting that $\nu(K_{2,n-2})=2$. This proves Part (i).

Part (ii) follows by an easy calculation.
\end{proof}

Let $G$ be a graph on $n\ge 2$ vertices without isolated vertices.
Note that the matching number of $G$ is an upper bound of the domination number of $G$.
So from \cite[Theorem 9]{GH} which states that the domination number is an upper bound for
$m_G[0,1)$, one gets the following lemma.

\begin{lemma} \label{med}
Let $G$ be a graph  without isolated vertices. Then $m_G[0,1)\le \nu(G)$.
\end{lemma}

For graphs with minimum degree at least two and one component different from $C_5$, we have a better bound than the one in Lemma \ref{med}.

\begin{theorem}\label{delta2}
Let $G$ be a graph with $\delta(G)\ge 2$. If $G\not\cong kC_5$ for any positive integer $k$, then $m_G[0,1)\le \nu(G)-1$.
\end{theorem}
\begin{proof}
As $\delta(G)\ge 2$, there exists a cycle in some component $G_0$ of $G$ that is not a $C_5$.
Let $C$ be a cycle in $G_0$ such that $G-V(C)$ has minimum number of isolated vertices.
Denote by $W$ the set of isolated vertices of $G-V(C)$.

Suppose first that $W=\emptyset$. If  $|V(C)|\ne 5$, then by Lemmas \ref{addedges}, \ref{bcn} and \ref{med},
\begin{align*}
m_G[0,1)&\le m_{C\cup (G-V(C))}[0,1)\\
&=m_{C}[0,1)+m_{G-V(C)}[0,1)\\
&\le \nu(C)-1+\nu(G-V(C))\\
&\le \nu(G)-1.
\end{align*}
Suppose that $|V(C)|=5$. Then there exists a component $H$ of $G_0-V(C)$ such that there is a vertex $v\in V(H)$ with $N_{G_0, C}(v)\ne \emptyset$. Assume that $v$ is such a vertex with least number of neighbors in $H$. Denote by $v_1$ one neighbor of $v$ in $C$.
By the choice of $v$, $H-v$ has at most one isolated vertex.

Suppose that $H-v$ has exactly one isolated vertex, say $v'$. Then, as $\delta(G)\ge 2$, $v$ is a neighbor of $v'$ and $v'$  has a neighbor, say  $v_1'$, on  $C$. Then we have another cycle $C'=vv'v_1'Pv_1v$, where $P$ denotes the longer path on $C$ from $v_1'$ to $v_1$ if $v_1\ne v_1'$, and $P=v_1=v_1'$ otherwise. Note that $|V(C')|\ne 5$ and $G-V(C')$ has no isolated vertices. Now the result follows by replacing $C$ by $C'$ and using the previous argument.

If $H-v$ has no isolated vertices, then $G[V(C)\cup V(H)]$ contains $P_6\cup (H-v)$ as a spanning subgraph, so we have by Lemmas \ref{addedges}, \ref{k2t} and \ref{med} that
\begin{align*}
m_G[0,1)&\le m_{G[V(C)\cup V(H)]}[0,1)+m_{G-V(C)-V(H)}[0,1)\\
&\le m_{P_6}[0,1)+m_{H-v}[0,1)+m_{G-V(C)-V(H)}[0,1)\\
&\le \nu(P_6)-1+\nu(H-v)+\nu(G-V(C)-V(H))\\
&\le \nu(G)-1,
\end{align*}
as desired.

Suppose next that $W\ne \emptyset$, say $w\in W$.
As $\delta(G)\ge 2$, $w$ has  at least two neighbors, say $w_1$ and $w_2$, on $C$. If $|V(C)|=3$, then
$C'=ww_1w_3w_2w$, with $w_3$ denoting  the  vertex on $C$ different from $w_1$ and $w_2$,
is a cycle of $G$ such that $G-V(C')$ has less isolated vertices than $G-V(C)$, a contradiction.
So $|V(C)|\ge 4$.
If $|V(C)|\ge 5$, or $|V(C)|=4$ with $w_1w_2\in E(C)$, then $C'=ww_1Qw_2w$,
 with $Q$ denoting a shortest path on $C$ from $w_1$ and $w_2$,
is a cycle of $G$ such that $G-V(C')$ has less isolated vertices than $G-V(C)$, a contradiction.
So $|V(C)|=4$ and no vertex in $W$ is adjacent to two consecutive vertices on $C$.
Let $C:=u_1u_2u_3u_4$, $W_1=\{w\in W: N_{G,C}(w)=\{u_1,u_3\}\}$ and $W_2=\{w\in W:N_{G,C}(w)=\{u_2,u_4\}\}$.
If $W_1\ne\emptyset$ and $W_2\ne \emptyset$, say $w_1\in W_1$ and $w_2\in W_2$, then $C':=u_1u_2w_2u_4u_3w_1u_1$ is a cycle of $G$ such that $G-V(C')$ has less isolated vertices than $G-V(C)$, a contradiction.
So $W_1=\emptyset$ or $W_2=\emptyset$. Assume that $W_2=\emptyset$. Then $G_1:=G[V(C)\cup W_1]$ contains $K_{2,2+|W|}$ as a spanning subgraph,  so by Lemmas \ref{addedges}, \ref{k2t} and \ref{med},
\begin{align*}
m_G[0,1)&\le m_{G_1}[0,1)+m_{G-V(G_1)}[0,1)\\
&\le \nu(G_1)-1+\nu(G-V(G_1))\\
&\le \nu(G)-1,
\end{align*}
as desired.
\end{proof}

Lemma \ref{k2t} shows the sharpness of the bound in Theorem \ref{delta2}.
Akbari el al. \cite[Theorem 3.4]{AADT} reported a similar result on the distribution of the Laplacian eigenvalues.
A  known result related to Theorem \ref{delta2} is:
For an $n$-vertex graph $G$ without isolated vertices, $m_G[0,2)\le n-\nu (G)$, see \cite{GH}.

Denote by $\alpha(G)$ the independence number of a graph $G$.
Ghodrati and Hosseinzadeh \cite{GH} noted that for a graph on $n$ vertices, $m_G[0,1)\le \alpha(G)$ and
$m_G[0, n-\alpha(G)]\ge \alpha (G)-1$.  As the  signless Laplacian version of \cite[Theorem 2.12]{Me}, we have

\begin{theorem} For a graph $G$ with $n$ vertices,
$\alpha(G)\le \min\{m_G[\delta(G), 2n-2], m_G[0, \Delta(G)]\}$, where $\Delta(G)$ is the maximum degree of $G$.
\end{theorem}

\begin{proof}
Assume  that $\{v_1,\dots, v_k\}$ is an independent set of $G$, where $k=\alpha(G)$. Let $B$ be the principal submatrix of
$Q(G)-\delta(G)I_n$ whose rows and columns are indexed by $v_1, \dots, v_k$. It is evident that all eigenvalues of $B$ are nonnegative. By Lemma \ref{interlacing},
$\rho_k(Q(G)-\delta(G)I_n)\ge \rho_k(B)\ge 0$, so $q_k(G)\ge \delta(G)$.
It follows that $m_G[\delta(G), 2n-2]\ge k$.
Similarly, $\rho_k(\Delta(G)I-Q(G))\ge 0$, so $m_G[0, \Delta(G)]\ge k$.
\end{proof}

As an immediate consequence of the previous theorem, we have: For a graph with minimum degree one,
$m_G[0,1)\le \alpha(G)\le m_G[1, 2n-2]$.

\subsection{Distribution of signless Laplacian eigenvalues and diameter}

An easily obtained relation between the distribution of signless Laplacian eigenvalues and diameter is:
For  an $n$-vertex connected  graph  with diameter $d$,
$m_G(2, 2n-2]\ge \lfloor \frac{d}{2}\rfloor$. It follows from the following easy result with a similar
version for Laplacian eigenvalues (see \cite[Theorem 2.15]{Me}).

\begin{theorem}
Let $G$ be  an $n$-vertex connected  graph  and $\ell$  the length of a longest path of $G$. Then
$m_G(2, 2n-2]\ge \lfloor \frac{\ell}{2}\rfloor$.
\end{theorem}

\begin{proof}
Let $P$ be a longest path of $G$. Evidently, $P$ is a tree. By \cite[Corollary 4.3]{GMS}, $m_P(2, 2\ell]=m_P(2,\ell+1]\ge \lfloor \frac{\ell}{2}\rfloor$.
Let
$H$ be the spanning subgraph of $G$ with edge set $E(P)$. By Lemma \ref{addedges},
$m_G(2, 2n-2]\ge m_H(2, 2n-2]=m_P(2,2\ell]$, so the result follows.
\end{proof}

The main result of this section is the following theorem.

\begin{theorem}\label{diameter-x}
Let $G$ be  an $n$-vertex connected  graph with diameter $d$. Then  $m_G[0,n-2)\ge d-1$. Moreover, if
$3\le d\le n-3$, then
\[
m_G[0,n-d+1)\ge \begin{cases}
d & \mbox{if $3\le d\le n-5$,}\\
d-1 & \mbox{otherwise.}
\end{cases}
\]
\end{theorem}

\begin{proof} The case when $d=1$ is trivial. Suppose that $d\ge 2$.
Let $P:=v_1\dots v_{d+1}$ be a diametral path of $G$. Then
\[
\delta(G)\le \delta_G(v_1)\le n-d.
\]
By \cite[Theorem 4]{GH}, $q_i(G)\le n-3$ whenever $\delta(G)+2\le i\le n-1$, so
$q_{n-d+2}(G)<n-2$, implying that $m_G[0,n-2)\ge d-1$.

Suppose that $3\le d\le n-3$.

\noindent
{\bf Case 1.} $d=n-3$.

It suffices to show $q_{5}(G)< 4$.

There are exactly two vertices, say $u$ and $v$  outside $P$ in $G$. As $P$ is a diametral path, $u$ ($v$, respectively) is adjacent to at most three consecutive vertices on $P$, and  $\max\{\delta_{G,P}(u), \delta_{G,P}(v)\}\ge 1$.
If $|N_{G,P}(w)|<3$ with $w\in \{u,v\}$, then, by Lemma \ref{addedges}, we may add edges between $w$ and proper vertices on $P$ so that $u$ ($v$, respectively) is adjacent to exactly three consecutive vertices on $P$ and $P$ remains to be a diametral path of the resulted graph.
Assume that $N_{G,P}(u)=\{v_{r-1},v_r,v_{r+1}\}$ and $N_{G,P}(v)=\{v_{s-1},v_s,v_{s+1}\}$ whether $u$ and $v$ are adjacent or not, where $2\le r,s\le d$. Assume that $r\le s$.
If $s-r=0$, then $G$ is a spanning subgraph of $G_{n,n-3,r}$, so we have $q_5(G)\le q_{5}(G_{n,n-3,r})<4$ by Lemmas \ref{addedges} and \ref{gndt}.
If $s-r=1$, then $G$ is a spanning subgraph of $G_{n,n-3,r,1}$, so we have $q_5(G)<4$ by Lemmas \ref{addedges} and \ref{gndra}.
Suppose that $s-r\ge 2$. Let $u_i=v_i$ for $i=1,\dots, r-1$, $u_r=u$, $u_{i+1}=v_i$ for $i=r,\dots,s-1$, $u_{s+1}=v$, and $u_{i+2}=v_i$ for $i=s,\dots,n-2$.
Then $G$ is the graph obtained from $P_n=u_1\dots u_n$ by adding edges $u_{r-1}u_{r+1}$, $u_ru_{r+2}$, $u_su_{s+2}$ and $u_{s+1}u_{s+3}$. Thus, by Lemmas \ref{addedges} and \ref{Cn},
\[
q_5(G)\le q_1(G-u_{r-1}u_{r+1}-u_ru_{r+2}-u_su_{s+2}-u_{s+1}u_{s+3})=q_1(P_n)<4.
\]

\noindent
{\bf Case 2.}  $d\le n-4$.

Let $B$ be the principal submatrix of $Q(G)$ corresponding to vertices of $P$.
Then $B=Q(P)+M$, where $M$ is a diagonal matrix in which the diagonal entry corresponding to vertex $z\in V(P)$ is $\delta_G(z)-\delta_{G,P}(z)$.
As there are exactly $n-d-1$ vertices outside $P$, $\delta_G(z)-\delta_{G,P}(z)\le n-d-1$.

Let $W=\{z\in V(P):\delta_G(z)-\delta_{G,P}(z)\ge n-d-2 \}$.
Let $\alpha=\min\{i: v_i\in W\}$ and $\beta=\max\{i:v_i\in W\}$. Let $N_\alpha=N_G(v_\alpha)\setminus V(P)$ and $N_\beta=N_G(v_\beta)\setminus V(P)$.

\noindent
{\bf Claim.} If $|W|\ge 2$, then  $\beta-\alpha\le 2$ and
$|N_\alpha\cap N_\beta|\ge n-d-3$.

Suppose that $|W|\ge 2$.
If $\beta-\alpha\ge 3$, then, as $P$ is a diametral path of $G$, one gets $N_G(v_\alpha)\cap N_G(v_\beta)=\emptyset$, so \[
2(n-d-2)\le (\delta_G(v_\alpha)-\delta_{G,P}(v_\alpha))+(\delta_G(v_\beta)-\delta_{G,P}(v_\beta))\le n-d-1,
\]
contradicting the fact that $d\le n-4$.
Thus $\beta-\alpha\le 2$.  As
 $|N_\alpha\cup N_\beta|\le |V(G)\setminus V(P)|=n-d-1$, one gets
\begin{align*}
|N_\alpha\cap N_\beta|&=|N_\alpha|+|N_\beta|-|N_\alpha\cup N_\beta|\\
&\ge 2(n-d-2)-(n-d-1)=n-d-3.
\end{align*}
This proves the claim.

\noindent
{\bf Case 2.1.}  $d=n-4$.

We want to show that $q_{6}(G)< 5$.

If $|W|\le 2$, then  $\rho_3(M)\le 1$, so by Lemmas \ref{interlacing}, \ref{cw} and \ref{Cn}, we have \[
q_{6}(G)=q_{n-(n-3)+3}(G)\le \rho_3(B)\le q_1(P)+\rho_3(M)<4+1=5.
\]

Suppose that $|W|\ge 3$. By the above claim, $\beta-\alpha\le 2$ and $|N_\alpha\cap N_\beta|\ge 1$.
It then follows that $\beta-\alpha=2$, i.e., $W=\{v_\alpha,v_{\alpha+1},v_\beta\}$.

If $|N_\alpha\cap N_\beta|\ge 2$, say $z_1,z_2\in N_{\alpha}\cap N_{\beta}$, then $H_1:=G[V(P)\cup \{z_1,z_2\}]$ is isomorphic to a spanning subgraph of $G_{n-1,n-4,\alpha+1}$.
Denote by $w$ the unique vertex in $V(G)\setminus V(H_1)$.
We have by Lemmas \ref{deletevertices}, \ref{addedges} and \ref{gndt} that
\[
q_6(G)\le q_{5}(G-w)+1=q_{5}(H_1)+1\le q_5(G_{n-1,n-4,\alpha+1})+1<4+1=5.
\]

Suppose that $|N_\alpha\cap N_\beta|=1$. Assume that $N_\alpha\cap N_\beta\subset N_G(v_{\alpha+1})$ by Lemma \ref{addedges}.
As $|V(G)\setminus V(P)|=3$, we have $|N_\alpha\cap N_{\alpha+1}|= 2$ or $|N_{\beta}\cap N_{\alpha+1}|= 2$, where $N_{\alpha+1}=N_G(v_{\alpha+1})\setminus V(P)$. Assume that $|N_\alpha\cap N_{\alpha+1}|= 2$. Denote by $u$ the unique vertex in $N_\alpha\cap N_{\alpha+1}$ that is not a neighbor of $v_\beta$.
It is possible that $u$ is adjacent to $v_{\alpha-1}$. Assume that $uv_{\alpha-1}\in E(G)$ by Lemma \ref{addedges}. So $H_2:=G[V(P)\cup (N_\alpha\cap N_{\alpha+1})]$ is isomorphic to a spanning subgraph of  $G_{n-1,n-4,\alpha,1}$.
Denote by $u'$ the unique vertex in $V(G)\setminus V(H_2)$.
It then follows from Lemmas \ref{deletevertices}, \ref{addedges} and \ref{gndra} that \[
q_6(G)\le q_5(G-u')+1= q_5(H_2)+1\le q_5(G_{n-1,n-4,\alpha,1})+1<4+1=5.
\]

\noindent
{\bf Case 2.2.}  $d\le n-5$.

We want to show that $q_{n-d+1}(G)<n-d+1$.

If $|W|\le 1$,
then  $\rho_2(M)\le n-d-3$, so we have by Lemmas \ref{interlacing}, \ref{cw} and \ref{Cn} that
\[
q_{n-d+1}(G)=q_{n-(d+1)+2}(G)\le \rho_2(B)\le q_1(P)+\rho_2(M)<4+n-d-3=n-d+1.
\]

Suppose that $|W|\ge 2$. By the above claim, $\beta-\alpha\le 2$ and $|N_\alpha\cap N_\beta|\ge n-d-3$.

Suppose that $\beta-\alpha=2$ or $\beta-\alpha=1$ with $N_\alpha\cap N_G(v_{\beta+1})= \emptyset$ and $N_\beta\cap N_G(v_{\alpha-1})=\emptyset$.
Then $G$ contains a subgraph $H$ induced by the set of vertices on $P$ together with $n-d-3$ vertices in $N_\alpha\cap N_\beta$ that is  isomorphic to a spanning subgraph of $G_{n-2,d,\alpha+1}$.  By Lemmas \ref{addedges} and \ref{gndt}, $q_{n-d-1}(H)\le q_{n-d-1}(G_{n-2,d,\alpha+1})<n-d-1$.
Let $w_1$ and $w_2$ be the  vertices in $V(G)\setminus  V(H)$.
Then, by Lemma \ref{deletevertices}, we have
 \[
q_{n-d+1}(G)\le q_{n-d-1}(G-w_1-w_2)+2=q_{n-d-1}(H)+2<n-d-1+2=n-d+1.
\]

Suppose next that $\beta-\alpha=1$ with $N_\alpha\cap N_G(v_{\beta+1})\ne \emptyset$ or $N_\beta\cap N_G(v_{\alpha-1})\ne \emptyset$, say $|N_\beta\cap N_G(v_{\alpha-1})|=c>0$. Then $G$ contains a subgraph $H_1$ induced by the set containing vertices on $P$, all vertices in $N_\beta\cap N_G(v_{\alpha-1})$  and $n-d-3-c$ vertices in $N_\alpha\cap N_\beta$ that is
isomorphic to a spanning subgraph of $G_{n-2,d,\alpha,c}$. Let $z_1$ and $z_1'$ be the  two vertices in $V(G)\setminus V(H_1)$. By Lemmas \ref{deletevertices}, \ref{addedges} and \ref{gndr}, we have
\begin{align*}
q_{n-d+1}(G)&\le q_{n-d-1}(G-z_1-z_1')+2\\
&=q_{n-d-1}(H_1)+2\\
&\le q_{n-d-1}(G_{n-2,d,\alpha,c})+2\\
&<n-d-1+2\\
&=n-d+1.
\end{align*}
This completes the proof.
\end{proof}

The bound in Theorem \ref{diameter-x} for $d=1,2$ is tight.
For $d=1$, $\sigma_Q(K_n)=\{2n-2,(n-2)^{[n-2]}\}$, so $m_{K_n}[0,n-2)=0$.
For $d=2$, $K_n-e$ is an $n$-vertex graph with diameter $2$ and $m_{K_n-e}[0,n-2)=1$  by Lemma \ref{Cn}.
The restriction on $d$ in
the second part can not be weakened as it is not true for an $n$-vertex graph $G$ with diameter $d\ge n-2$. This is because $G\cong P_n$ and $m_{P_n}[0,2)<n-2=d-1$ for $n\ge 6$ if $d=n-1$, and $G$ contains $P_{n-1}$ as an induced subgraph and $m_G[0,2)\le m_{P_{n-1}}[0,2)<n-3=d-1$ for $n\ge 9$ if $d=n-2$.

By Theorem \ref{diameter-x}, if $d\le n-5$, then $m_{G_{n,d,t}}[0,n-d+1)\ge d$. However, if $3\le t\le d-1$, then $G_{n,d,t}$ has exactly $d-1$ Laplacian eigenvalues in the interval $[0,n-d+1)$, see \cite{XZ}.

The first part of Theorem \ref{diameter-x} can be improved for graphs with diameter three.

\begin{theorem}\label{dia3}
Let $G$ be an $n$-vertex connected graph with diameter $3$, where $n\ge 7$. Then $m_G[0,n-3)\ge 2$. Equality holds if $G\cong G_{n,3,2}$ or $G\cong G_{n,3,2,a}$ with $1\le a\le n-5$.
\end{theorem}
\begin{proof}
As $G$ is a spanning subgraph of $G_{n,3,2}$ or $G_{n,3,2,a}$, and $q_{n-1}(G_{n,3,2,a})<n-3=q_{n-2}(G_{n,3,2,a})$ by Lemma \ref{Cn}, it suffices to show that $q_{n-1}(G_{n,3,2})<n-3=q_{n-2}(G_{n,3,2})$.

Let $H=G_{n,3,2}$.
As $(n-3)I_n-Q(H)$ has $n-3$ equal rows, $n-3$ is an eigenvalue of $Q(H)$ with multiplicity at least $n-4$. Let $v_1\dots v_4$ be a diametral path of $H$ so that $\delta_H(v_4)=1$.
With respect to the  partition $V(H)=\{v_1\}\cup N_H(v_1)\cup \{v_3\}\cup \{v_4\}$, $Q(H)$ has an equitable quotient matrix $B$ with \[
B=\begin{pmatrix}
n-3&n-3&0&0\\
1&2n-6&1&0\\
0&n-3&n-2&1\\
0&0&1&1
\end{pmatrix}.
\]
Let $f(x)=\det(xI_4-B)$. As \[
f(n-2)=-(n-4)^2<0,
\]
\[
f(n-3)=(n-3)^2>0
\]
and \[
f(2)=-2(n-3)(n-4)(n-6)<0,
\]
we have $\rho_4<2<\rho_3<n-3<\rho_2<n-2<\rho_1$, where $\rho_i=\rho_i(B)$ for $i=1,\dots, 4$.
By Lemma \ref{quo}, $\sigma_Q(H)=\{\rho_1,\rho_2,(n-3)^{[n-4]}, \rho_3,\rho_4\}$. So $q_{n-2}(H)=n-3$ and $q_{n-1}(H)=\rho_3<n-3$.
\end{proof}

\bigskip
\bigskip

\noindent {\bf Acknowledgement.}
%The authors  would like to thank the editor and the referee for constructive comments and suggestions.
This work was supported by the National Natural Science Foundation of China (No.~12071158).

\end{document}